\documentclass[12pt,%reqno
]{amsart}
\usepackage{amssymb,mathrsfs}
\usepackage{color,umoline}

\def\norm#1{\|#1\|}

\newcommand{\Hl}{\mathcal{H}}
\newcommand{\ball}{\mathbb{S}}

\newcommand{\R}{\mathbb{R}}

\newcommand{\supp}{\operatorname{supp}}

\newcommand{\les}{\lesssim}

\newcommand{\Del}[1]{}

\numberwithin{equation}{section}

\newtheorem{thm}{Theorem}[section]
\newtheorem{cor}[thm]{Corollary}
\newtheorem{lem}[thm]{Lemma}

\theoremstyle{remark}
\newtheorem{rem}{Remark}

\begin{document}

\title[Estimate for the adjoint restriction operator]{A Sobolev estimate for the adjoint \\ restriction operator}

\author[Y. Cho]{Yonggeun Cho}
\author[Z. Guo]{Zihua Guo}
\author[S. Lee]{Sanghyuk Lee}

\address{Yonggeun Cho, Department of Mathematics, and Institute of Pure and Applied Mathematics, Chonbuk National University, Jeonju 561-756, Republic of Korea}
\email{changocho@jbnu.ac.kr}
\address{Zihua Guo, School of Mathematical Sciences, Peking
University, Beijing 100871, China} \email{zihuaguo@math.pku.edu.cn}
\address{Sanghyuk Lee, Department of Mathematical Sciences, Seoul National University, Seoul 151-747, Republic of Korea}
\email{shklee@snu.ac.kr}

\subjclass[2010]{42B10, 42B25} \keywords{Restriction estimate,
Strichartz estimate}

\begin{abstract} In this note we consider the adjoint restriction estimate
for hypersurface under additional regularity assumption. We  obtain
the optimal $H^s$-$L^q$  estimates  and their mixed norm
generalizations. As applications we prove some weighted Strichartz
estimates for the propagator $\varphi\to e^{it(-\Delta)^{ \alpha/
2}}\varphi$, $\alpha>0$.
\end{abstract}

\maketitle

\section{Introduction}
The Fourier extension operator (the adjoint of restriction operator)
$R^*\!$ for the sphere is defined by
\[R^*\!f(x) =\int_{\ball^{d}}e^{ix\cdot\xi}f(\xi)d\sigma(\xi), \quad  x\in \mathbb R^{d+1}.\]
Here $\ball^{d}$ denotes the unit sphere in $\mathbb R^{d+1}$ and
$d\sigma$ is the induced Lebesgue measure on $\ball^{d}$. The
problem which is known as the restriction problem for the sphere is
to determine the range of $p,$ $q$ for which
\begin{align}\label{eq:Epq}
\norm{R^*\!f}_{L^q(\mathbb R^{d+1})}\leq C \norm{f}_{L^p(\ball^d)}
\end{align}
holds. As can be easily seen by Knapp's example and the asymptotic
expansion of $\widehat{d\sigma}$, \eqref{eq:Epq} holds only if
$q>\frac{2(d+1)}{d}$ and
\[(d+2)/q\le d(1-1/p).\]
 When $d=1$, \eqref{eq:Epq} on the optimal range was obtained by
Zygmund \cite{zy} (see \cite{fef} for an earlier result due to
Fefferman and Stein \cite{fef}). It has been conjectured that the
necessary condition is sufficient for \eqref{eq:Epq} in higher
dimensions but it still remains open.
When $p=2$, the sharp boundedness is due to Tomas \cite{to} and
Stein \cite{st1}. The result beyond Stein-Tomas range was first
obtained by Bourgain when $d=2$, and further progresses were made by
the works of  Wolff \cite{w1}, Tao, Vargas and Vega \cite{tvv}, Tao
and Vargas \cite{tv1},  and Tao \cite{ta1}. (Also see \cite{w2, lee,
v, lv} for results regarding different types of hypersurfaces.)
Recently, Bourgain and Guth \cite{bogu} improved the range.
Especially, when $p=q$, the estimate \eqref{eq:Epq} was shown  to be
true for $p\in(56/17,\infty]$ in $\mathbb R^{2+1}$ (see
\cite[p.1265]{bogu}) and for $p\in(p_{\circ}(d),\infty]$ in $\mathbb
R^{d+1}$, $d\ge 3$ where $p_{\circ}(d)=2+ 12/(4d+1-k)$ if $
d+1\equiv k ~(mod~ 3),\ k=-1, 0,1$. Their result relies on the
multilinear restriction estimate in Bennett, Carbery and Tao
\cite{BCT}.
A further improvement in $\mathbb R^{2+1}$ which gives
\eqref{eq:Epq} for $p=q\in (3.25, \infty]$ was obtained by Guth
\cite{Guth}.

In this note, we consider the estimate \eqref{eq:Epq} from a
different perspective. Let us denote by $\Hl^s$ the $L^2$ Sobolev
space of order $s$ on the sphere. The main purpose of this paper is
to study the restriction estimate
\begin{equation}
\label{hs}
 \norm{R^*\!f}_{q}\leq C
\norm{f}_{\Hl^{s}}
\end{equation}
and to find the optimal range of $s,q$ for which \eqref{hs} holds.
When $s=0$, by the necessary condition and Tomas-Stein theorem
\eqref{hs} holds if and only if $q\ge 2(d+2)/d$. It is natural to
expect that the range of $q$ gets wider if  $f$ is assumed to have
an additional regularity, that is to say $s>0$. However, for $q\le
{2(d+1)}/{d}$ this estimate fails because $\widehat {d\sigma}\not
\in L^q$. Hence, the estimate \eqref{hs} is of interest for $q$
satisfying  $ {2(d+1)}/{d}<q< 2(d+2)/d$. By the Knapp type example
again, it can be shown (see the paragraph below Theorem \ref{xed})
that \eqref{hs} is possible only if
\begin{equation}\label{nec}
s\ge s_q= s_q(d)=:\frac{d+2}{q}-\frac{d}{2}.
\end{equation}

The estimate \eqref{hs} for $s=s_q$  can be deduced from the sharp
restriction estimate (\eqref{eq:Epq} with $(d+2)/q= d(1-1/p)$) by
making use of the embedding $\Hl^{\alpha}(\ball^{d}) \subset
L^p(\ball^{d})$ with $\alpha\ge d/{2}-d/{p}$ and $p\neq \infty$.
Since we have the sharp restriction estimate for $d=1$, we get
\eqref{hs} for $4<q<6$ with the optimal regularity. In higher
dimensions it seems natural to expect that the estimate \eqref{hs}
holds whenever ${2(d+1)}/{d}<q< 2(d+2)/d$ and \eqref{nec} is
satisfied. As it turns out, this is indeed the case and the problem
is much easier than the restriction estimate \eqref{eq:Epq}. This is
mainly  due to the fact that the inequality \eqref{hs} is based on
$L^2$-spaces.

\medskip

The following is our first result.

\begin{thm}\label{thm-sphere}
Let $d\geq1$ and $q<{2(d+2)}/{d}$. Then \eqref{hs} holds if and only
if $s\ge s_q$ and $q>2(d+1)/d$.
\end{thm}

\subsection*{\it Mixed norm generalization}
We now consider a more general class of operators. Let $\phi$ be a
smooth function. Then  we define an extension operator $ E$ by
\[E f(x,t)= E_\phi f(x,t)= \int e^{i(x\cdot \xi+t\phi(\xi))} a(\xi) f(\xi) d\xi,
\  \  (x,t)\in \mathbb R^d\times \mathbb R, \]
where $a$ is a smooth function supported in $B(0,1)$. We denote by
$B(z,r)$ the ball of radius $r$ which is centered at $z$.
Generalizing \eqref{hs} we consider the estimate
 \begin{equation}\label{mixed} \|E f\|_{L_t^q(\mathbb R,\,
L_x^r(\mathbb R^d))}\le C\|f\|_{H^{s}}.
\end{equation}
Here $H^s$ is the usual inhomogeneous Sobolev space of order $s$. If
the hessian matrix $H\phi$ of $\phi$ is nonsingular on the support
of $a$, then by the Strichartz estimate \eqref{mixed} holds for
$s=0$ and $q,r$ satisfying $2\leq q,r\leq \infty$, $d/r+2/q\le d/2$,
$(q,r,d)\neq (2,\infty,2)$ (see \cite{kt}). Let us set
\begin{equation}\label{sc} s_c=s_c(r,q,d)=\frac{d}r+\frac
2q-\frac{d}2\,.\end{equation}

As a mixed norm generalization of \eqref{hs} we obtain the
following.

\begin{thm}\label{xed} Suppose $\det H \phi\neq 0$ on the support of $a$.
If $2\le q,r<\infty$,  $d/r+2/q > d/2$ and $d/r+1/q < d/2$, then
\eqref{mixed} holds whenever $s\ge s_c$.
 If $d/r+1/q = d/2$ and $q\neq 2$, then the weak type
estimate $\| E f\|_{L_t^{q,\infty}(\mathbb R,\, L_x^r(\mathbb
R^d))}\le C\|f\|_{H^{s_c}}$ holds.
\end{thm}

Theorem \ref{thm-sphere} is an obvious consequence of Theorem
\ref{xed} as this can be easily shown by making use of proper
coordinate patches (for example, see \cite[Sec. 3, Ch. 4]{taylor}).
The estimate \eqref{mixed} is no longer true if $d/r+1/q\ge d/2$
because $ E_\phi(1)\not\in L^q_tL^r_x$, which follows from Lemma
\ref{asymp}. And the regularity condition is also optimal since
\eqref{mixed} fails if $s<s_c$. Indeed, let us take
$f(\xi)=\lambda^{\frac{d}2}\eta(\lambda\xi)$ with a  compactly
supported smooth $\eta$  so that $\|f\|_{H^s}\lesssim \lambda^s$.
Then it is easy to see that $|Ef(x,t)|\gtrsim \lambda^{-\frac{d}2}$
if $|t|\le c\lambda^2$ and $|x+\nabla \phi(0)t|\leq c\lambda$ for a
small $c>0$. Hence \eqref{mixed} implies
\[ \lambda^{-\frac{d}2}\lambda^{\frac{d}r+\frac 2q}\lesssim \lambda^s. \]
Letting $\lambda\to\infty$, we get the condition \eqref{nec}.

%%%%
\begin{rem}\label{remark1}
Since $E$ is a localized operator,  in Theorem \ref{xed}
we may replace $H^{s}$ by the homogeneous Sobolev space $\dot H^{s}$ provided that $0<s<d/2$.
The same is also true for Theorem \ref{thm-sphere}. Indeed, it is easy to see that  for $0<s<d/2$
\[\norm{f\beta}_{H^{s}}\les \norm{f}_{\dot H^s},\]
where $\beta\in C_0^\infty(\R^d)$.
\end{rem}

\subsubsection*{Application to Strichartz estimates} The adjoint
restriction estimates are closely related to the Strichartz
estimates. Especially, various space time estimates for the
Schr\"odinger operator can be deduced from restriction estimates for
the paraboloid. (See \cite{lrv} for a detailed discussion.) As
applications of Theorem \ref{xed} we obtain some weighted Strichartz
estimates on the range where the usual Strichartz estimates are not
allowed.

Let $\alpha>0$ and $\alpha\neq 1$. We first consider the  weighted Strichartz estimates for
the fractional Schr\"odinger operator
\begin{align}\label{w-str}
\|e^{it(-\Delta)^{ \alpha/ 2}}\varphi\|_{L_t^q(\mathbb R,\,
L_x^r(\mathbb R^d))}\lesssim
\||x|^{\mu}(-\Delta)^{\nu/2}\varphi\|_{2}.
\end{align}

\begin{cor}\label{str}
Let $d\ge 1$,  $\alpha > 0$, $\alpha \neq 1$ and $2\le q,r<\infty$. If  $d/r +
2/q
> d/2$ and $d/r + 1/q < d/2$, then \eqref{w-str} holds
provided that $\mu=s_c$ and $\nu=\frac{2-\alpha}{q}$.
\end{cor}

We now consider the case $\alpha=1$, that is to say the wave
operator. Let us set $s_c^w:=s_c(q,r, d-1).$

\begin{cor}\label{str2} Let $d\ge 2$,
$\alpha=1$ and $2\le q,r<\infty$. If $(d-1)/r + 2/q > (d-1)/2$ and
$(d-1)/r + 1/q < (d-1)/2$, then \eqref{w-str} holds provided that
$\mu=s_c^w$ and $\nu=\frac{1}{2}+\frac{1}{q}-\frac{1}{r}$.
\end{cor}

Let $\Delta_{\omega}$ be the Laplace-Beltrami operator on the unit
sphere $\mathbb S^{d-1}\subset \mathbb R^{d}$ which is given by $
\Delta_{\omega}=\sum_{1\le i<j\le d} \Omega_{i,j}^2,$ $\Omega_{i,j}=
\omega_i\partial_j -\omega_j
\partial_i.$ Then define a Sobolev norm
$\|\cdot\|_{H^{\nu}_{sph}}$ by setting
\[\|f\|_{H^{\nu}_{sph}}^2= \int_0^\infty \int_{S^{d-1}}|(1-\Delta_{\omega})
^{\nu/2}  f(r\omega)|^2\,d\sigma_\omega r^{d-1}\,d r.\] Let us
consider the estimate
\begin{equation}\label{angular}
\|(-\Delta)^{{\gamma_1}/2}
e^{it\sqrt{-\Delta}}\varphi\|_{L_t^q(\mathbb R,\, L_x^r(\mathbb
R^d))} \le C\|\varphi\|_{H^{\nu}_{sph}}
\end{equation} assuming the natural scaling invariant condition
\begin{equation}\label{eq:scaling}
\gamma_1=\frac{1}{q}+\frac{d}{r}-\frac{d}{2}.
\end{equation}
This type of inequality was studied by  Sterbenz \cite{ster} to
extend the range of admissible $q,r$ by making use of angular
regularity (see \cite{jwy,cl, glnw} for related results and
references therein). It is known (\cite{ster}) that the estimate
\eqref{angular} holds only if
\[ \nu\ge s_c^w, \quad \frac{d-1}{r}+\frac 1q< \frac{d-1}2\,.\]
Sterbenz \cite{ster} showed \eqref{angular} for $\nu>s_c^w$. Our
result enables us to obtain the  estimate of the endpoint regularity
when $q\ge r$.

\begin{cor}\label{str3} Let $d\ge 2$ and $2\le q,r<\infty$.
If $q\ge r$, $(d-1)/r + 2/q > (d-1)/2$ and $(d-1)/r + 1/q <
(d-1)/2$, then \eqref{angular} holds provided that $\nu\ge s_c^w$
and \eqref{eq:scaling}.
\end{cor}

In addition to \,$\widehat{}$\, and $ {}^\vee$,  we use $\mathcal F,
\mathcal F^{-1}$ to denote Fourier, inverse Fourier transforms,
respectively. The paper is organized as follows. In section 2 we
provide a few preliminaries for the proof of Theorem \ref{xed}. In
section 3 we prove Theorem \ref{xed} and the proofs of Corollary
\ref{str}, \ref{str2} and \ref{str3} are given in section 4.

\section{Preliminaries}

For the proof of the estimate \eqref{mixed}  we may assume that
$\phi$  is close to a quadratic form. More precisely, let $\phi$ be
a smooth function satisfying that $\det H\phi$ is nonsingular. Then
we may assume that
\begin{equation}\label{normal} \phi(\xi)=\frac12
\xi^t D\xi+  \mathcal E(\xi), \,\, \| \mathcal E\|_{C^L (B(0,2))}\le
C \epsilon_0
\end{equation} with  a sufficiently small $\epsilon_0>0$ and  a
sufficiently large positive integer $L$ where $D$ is the diagonal
matrix with nonzero entries $\pm 1$.

\subsection*{\it Parabolic rescaling}

 Indeed, let $\xi_0$ be
a point in $B(0,1)$. By decomposing $a$ into finite number of smooth
functions which are supported in small balls we need only to
consider the localized operator
\[\int e^{i(x\cdot \xi+t\phi(\xi))}a_{\xi_0,\epsilon_0}(\xi) f(\xi) d\xi,\]
where $a_{\xi_0,\epsilon_0}$ is a smooth function supported in $
B(\xi_0,\epsilon_0)$.
 By Taylor expansion we have
\[ \phi(\xi)=\phi(\xi_0)+\nabla \phi(\xi_0)\cdot(\xi-\xi_0)+\frac12  (\xi-\xi_0)^t H \phi(\xi_0)(\xi-\xi_0)+O(|\xi-\xi_0|^3).\]
By discarding harmless factors, translation $\xi\to \xi+\xi_0$ and
the linear change of variables $ (x,t)\to (x+t\nabla \phi(\xi_0),t)$
we may assume
\[ \phi(\xi)=\frac12  \xi^t H \phi(\xi_0)\xi+O(|\xi|^3),\]
and then making a linear change of variables for both $x$ and $\xi$
we may further simplify $\frac12  \xi^t H \phi(\xi_0)\xi$ to the
form $\frac12(\xi_1^2\pm \xi_2^2\pm \dots \pm \xi_d^2)=\frac12 \xi^t
D\xi$ (diagonalization and rescaling). These operations do not
affect the estimate \eqref{mixed}  except  changing the constant
$C$. Now one can make the effect of error term small by further
scaling
\[\xi \to
\epsilon_0 \xi,\quad (x,t)\to (\epsilon_0^{-1} x, \epsilon_0^{-2} t)
\] which changes $x\cdot \xi+t(\frac12
\xi^t D\xi+O(|\xi|^3))$ to  $x\cdot \xi+t(\frac12 \xi^t
D\xi+O(\epsilon_0|\xi|^3))$. Hence we get \eqref{normal}.

\subsection*{\it Asymptotic of oscillatory integral} From the assumption
\eqref{normal} $\nabla\phi(\xi)$ is close to $D\xi$. Hence, with a
sufficiently small  $\epsilon_0$  we may assume that $\xi\to
\nabla\phi(\xi)$ is a diffeomorphism on $B(0,2)$ such that there is
a unique smooth function $\eta:B(0,15/8)\to \mathbb R^d$ such that
\[\nabla \phi(\eta( x))=-x.\]
Then we define
\begin{equation}
\label{psi} \psi( x)= x\cdot\eta(x)+\phi(\eta(x)).
\end{equation}
Since $|\nabla_\xi(\xi\cdot x+t\phi(\xi))|\gtrsim |x|$ if $|x|\ge
5t/4$ (here we are assuming $0<\epsilon_0\ll 1$), by routine
integration by parts we see that for any $M>0$
\begin{equation}\label{easy}\Big| \int e^{i( x\cdot \xi+ t\phi(\xi))} a(\xi)
d\xi\Big|\lesssim (1+|x|)^{-M}(1+|t|)^{-M}.\end{equation}  For the
other case $|x|< 5t/4$ we need the following which can be shown by
the stationary phase method. It is a special case of  Theorem 7.7.6
in H\"ormander \cite{Hor} (see p.222).

\begin{lem}\label{asymp} Suppose that $\phi$ is given by \eqref{normal} and $\supp a\subset B(0,1)$.
Then, if $t\ge 1$ and $|x|\lesssim t$, for a positive integer $N<
L/2-1$
\[\int e^{i( x\cdot \xi+ t\phi(\xi))} a(\xi) d\xi= \sum_{l=0}^N
t^{-\frac d2-l} e^{it\psi(\frac x t)}a_l (\frac x t)
+O(|t|^{-N-\frac d2-1}),\] where $a_l$ is a  bounded smooth function
with compact support.
\end{lem}

\section{Proof of Theorem \ref{xed}}
We first prove the estimate \eqref{mixed} and show the weak type
endpoint estimate at the end of this section.

\subsubsection*{Proof of \eqref{mixed}} To begin with we assume that the operator $E$ is
defined by $\phi$ which satisfies \eqref{normal} with a small
$\epsilon_0>0$ and a large $L$.  By time reversal symmetry it is
sufficient to show
\[ \norm{ E ( f)}_{{L_t^q((0,\infty),\, L_x^r(\mathbb R^d))}}\les
\norm{f}_{H^{s_c}}.\]
From the Strichartz estimate and Plancherel's
theorem we recall the estimate
\[ \norm{ E ( f)}_{{L_t^q((0,T),\, L_x^r(\mathbb R^d))}}\les
\norm{f}_2\] which holds for $T>0$ and $q,r$ satisfying $d/r+2/q=
d/2$, $(r,q,d)\neq (\infty, 2,2)$. By Plancherel's theorem and
H\"older's inequality we also have $\norm{ E f}_{{L_t^q((0,T),\,
L_x^2(\mathbb R^d))}}\les T^\frac1q \norm{f}_2$ for $q\ge 1$.
Interpolation between these two estimates gives
\begin{equation}\label{interpol}\norm{ E ( f)}_{{{L_t^q((0,T),\, L_x^r(\mathbb R^d))}}}\les
T^{\frac12(\frac dr +\frac 2q-\frac{d}2)} \norm{f}_2\end{equation}
for $q,r\ge 2$ satisfying $d/r+2/q\ge d/2$. Note that $s_c> 0$ if
$d/r+2/q > d/2$ and $d/r+1/q< d/2$.  Hence, we obviously need only
to show that
\[ \norm{ E ( f)}_{{L_t^q((1,\infty),\, L_x^r(\mathbb R^d))}}\les
\norm{f}_{H^{s_c}}.\]
From now on we assume that $t\ge 1$.

\medskip

Let $\beta\in C_c^\infty (1/2,2)$ such that $\sum_{-\infty}^\infty
\beta( 2^{-k} \rho)=1$ for $\rho>0$ and let us denote by $P_k$ the
Littlewood-Paley projection operator which is given by $\mathcal F(
P_k f)=\beta( 2^{-k} |\cdot|) \widehat f$. And we also set $\beta_0
= 1- \sum_{1}^\infty  \beta( 2^{-k} \rho)$ and define $P_{\le 0}$ by
$\mathcal F({P_{\le 0}f}) = \beta_0(|\cdot|) \widehat f$.

Using the projection operators, we decompose $Ef$  so that
\begin{equation}\label{decomp1}
 E f= E P_{\le 0} f + \sum_{k=1}^\infty E P_k f.
 \end{equation} It is easy
to handle $ E P_{\le 0} f$. Let us set
\[K(x,t)=\int e^{i(x\cdot \xi+t\phi(\xi))} a(\xi) d\xi.\]
By Fourier inversion we write
\[  E P_{\le 0} f(x,t)=\int K(x-y,t) \beta_0(|y|) \,f^\vee (y) dy.\]
Then  by \eqref{easy} and Lemma  \ref{asymp}  $|K(x,t)|\le
t^{-\frac{d}{2}}\chi_{B(0,\frac 54)}(\frac {x}t)+
(1+|x|)^{-M}t^{-M}$ for $N>0$. Hence it follows that
\[| E P_{\le 0} f(x,t)|\le C\int \Big( t^{-\frac{d}{2}}\chi_{B(0,\frac54)}(\frac {x-y}t)+
(1+|x-y|)^{-M}t^{-M}\Big) |\beta_0(|y|)  f^\vee (y)| dy.\] Since
$\beta_0(|\cdot|)$ is supported in $B(0,2)$, by Cauchy-Schwarz inequality and
Plancherel's theorem we see that
\begin{equation}\label{pw}\begin{aligned} | E P_{\le0} f(x,t)|&\le
Ct^{-\frac d2} (\chi_{\{|x|\le \frac54 t+2\}}+ (1+|x|)^{-M})
\|\beta_0(|\cdot|) f^\vee\|_1 \\ &\le Ct^{-\frac d2} (\chi_{\{|x|\le
\frac54 t+2\}}+ (1+|x|)^{-M})\|f\|_2.
\end{aligned}\end{equation}
So, by taking integration it follows that \[ \|E P_{\le0}
f\|_{{L_t^q((1,\infty),\, L_x^r(\mathbb R^d))}}\le  C\|t^{-\frac
d2+\frac d r}\|_{{L_t^q(1,\infty)}} \| f\|_2\le C\|f\|_2\] because
$d/r+1/q<d/2$.

\

For $k\ge 1$ and a large constant $C>0$, we set
\[\chi_k^\circ(t)=\chi_{[1, C2^{k}]}(t),
\quad \chi_{k}^c(t)=\chi_{[C2^{k},\infty)}(t).
\]
 We break the sum in \eqref{decomp1} so that
\[\sum_{k=1}^\infty  E P_k f =
\sum_{k=1}^\infty \chi_{k}^\circ  E P_k f + \sum_{k=1}^\infty
\chi_{k}^c  E P_k f.\] The contribution from the first summation is
rather easy to handle. In fact, since
$\frac{d}{r}+\frac2q-\frac{d}2>0$, using \eqref{interpol}
\begin{align*}
&\Big\|\sum_{k=1}^\infty \chi_{k}^\circ  E P_k f \Big\|_{L^q_t
L_x^r}\le \sum_{k=1}^\infty \|\chi_{k}^\circ  E P_k f \|_{L^q_t
L_x^r}\le C\sum_{k=1}^\infty 2^{\frac
k2(\frac{d}{r}+\frac2q-\frac{d}2)} \|P_k f \|_{L^2}
\\ &\le C\Big(\sum_{k=1}^\infty 2^{-k(\frac{d}{r}+\frac2q-\frac{d}2)}
\Big)^\frac12\Big( \sum_{k=1}^\infty
2^{2k(\frac{d}{r}+\frac2q-\frac{d}2)} \|P_k f
\|_{L^2}^2\Big)^\frac12 \le C\|f\|_{H^{s_c}}.
\end{align*}
Hence we are reduced to showing
\begin{equation}\label{reduced}
\Big\|\sum_{k=1}^\infty \chi_{k}^c  E P_k f \Big\|_{L^q_t L_x^r} \le
C\|f\|_{H^{s_c}}.
\end{equation}

We now use the asymptotic expansion in Lemma \ref{asymp}.  Let $A$
be a smooth function supported in $[-3/2,3/2]$ such that
 $A=1$ on  $[-5/4,5/4]$. Then we use   Lemma
\ref{asymp} for  $K(x,t)A(\frac{|x|}t)$ and \eqref{easy} for  $K(x,t)(1-A(\frac{|x|}t))$ to get
\[K(x,t)=\sum_{l=0}^N
t^{-\frac d2-l} e^{it\psi(\frac x t)}A_l(\frac x t)+ e (x,t),\]
where $e(x,t)=O((1+|x|+t)^{-N-\frac{d}{2}-1})$ and $A_l$ is a
smooth function supported in $B(0,3/2)$.  For simplicity we set
\[\widetilde A(\frac x t)= \sum_{l=0}^N
t^{-l} A_l(\frac x t).\]
We now define
\[   \widetilde E f(x,t)=t^{-\frac d2}\int   e^{it\psi(\frac {x-y} t)}\widetilde A(\frac {x-y} t)  \,f^\vee (y) dy, \quad  \mathcal R f(x,t)=\int   e(x-y,t)  \,f^\vee (y) dy.\]
Since $Ef=\int K(x-y,t)f^\vee (y) dy$, clearly $Ef= \widetilde
Ef+\mathcal Rf$. The left hand side of \eqref{reduced} is bounded by
\[\Big\|\sum_{k=1}^\infty \chi_{k}^c  \widetilde E P_k f \Big\|_{L^q_t L_x^r}+
\Big\|\sum_{k=1}^\infty \chi_{k}^c \mathcal R P_k f \Big\|_{L^q_t
L_x^r}.\] The contribution from $\sum_{k=1}^\infty \chi_{k}^c
\mathcal R P_k f$ is easy to control. In fact, since $q,r\ge 2$,
with a sufficiently large $N$ (using $e(x,t)=O((1+|x|+t)^{-N-1})$)
and by Young's inequality we see that for $s>0$
\begin{align*}
& \Big\|\sum_{k=1}^\infty \chi_{k}^c  \mathcal R P_k f \Big\|_{L^q_t
L_x^r} \le\sum_{k=1}^\infty \Big\| \chi_{k}^c \mathcal R P_k f
\Big\|_{L^q_t L_x^r}
\\
\le C &\sum_{k=1}^\infty  \Big\| t^{-1} \Big\|  \int (1+|\cdot
-y|)^{-N} |\beta(2^{-k}|y|) f^\vee (y)| dy \Big\|_{L^r_x}
\Big\|_{L_t^q(1,\infty)}
\\
\le  C& \sum_{k=1}^\infty \|\beta(2^{-k}|\cdot|)f^\vee\|_2\le C
\|f\|_{H^s}.
\end{align*}
For the last inequality we use the Cauchy-Schwarz inequality and
Plancherel's theorem. To get the desired bound, by multiplying
harmless factor $e^{-it\psi(x/t)}$  it is sufficient to consider the
operator $\widetilde { E_\psi}$ which is defined by
\begin{equation*}
\widetilde { E_\psi}f(x,t) = e^{-it\psi(\frac x t)} \widetilde {
E}f(x,t).
\end{equation*}
The estimate \eqref{reduced} follows from
\begin{equation}\label{aaa}
\Big\|\sum_{k=1}^\infty  \chi_k^c(t) \widetilde { E_\psi}
P_kf\Big\|_{L^q_{t}L_x^r} \le C \|f\|_{H^{s_c}}.
\end{equation}

Let us set \begin{equation} \label{mint} m(k,y, \xi)=2^{kd}\int
e^{i(t\psi(x-\frac{2^ky}{t})-t\psi(x)-x\cdot \xi)} \widetilde
A(x-\frac{2^ky}{t}) dx.\end{equation}
 Then by scaling $y\to 2^k y$ we get
\begin{equation}\label{ftrans} \mathcal F_x (\widetilde { E_\psi} P_kf( t \cdot , t)
)(\xi)= t^{-\frac d2} \int m(k,y, \xi)  \beta(|y|) f^\vee(2^k y)
dy.\end{equation}
Here $\mathcal F_x$ denotes the Fourier transform in $x$. In order
to get \eqref{aaa} we need the following lemma which shows that if
$t\gg 2^{k}$, the Fourier transform of $\widetilde { E_\psi}
P_kf(t\cdot,t)$ is essentially supported in the set $\{\xi:|\xi|\sim
2^k\}$.

\begin{lem} Let $1/2\le |y| \le 2$. If  $ |\xi|\ge B 2^k$ or $ |\xi|\le B^{-1} 2^k$ for some large
$B>0$, then  for any $M>0$ and multi-index  $\alpha$
 \begin{equation} \label{est}
 |\partial_\xi^\alpha m(k,y, \xi)|
\le C(\max\{2^k, |\xi|\})^{-M}\end{equation}
 with $C$ independent of $k,$ $y$.
\end{lem}

\begin{proof}  To see this, we consider the phase function of the
integral in \eqref{mint}
 \[t\psi(x-\frac{2^ky}{t})-t\psi(x)-x\cdot
 \xi.\]
 From \eqref{psi} and \eqref{normal} we have $\nabla \psi(x)=\eta(x)=Dx+ \mathcal E(x)$
where $\| \mathcal E\|_{C^{L}(B(0,2))}\lesssim \epsilon_0$. The
Hessian matrix of $\psi$ is close to the matrix $D$.
 Since $|y|\sim 1$ and $\frac {2^k}t\ll 1$, it is easy to see \[\Big|\nabla
\psi(x-\frac{2^ky}{t})-\nabla \psi(x)\Big|= \Big|M  \frac{2^ky}t+
O(\epsilon_0 \frac{2^ky}t)\Big|\sim  \frac {2^k}t.\] Since  $
|\xi|\ge B 2^k$ or $ |\xi|\le B^{-1} 2^k$ for some large $B>0$, we
get
\[ |\nabla_x(t\psi(x-\frac{2^ky}{t})-t\psi(x)-x\cdot
 \xi)|\gtrsim \max(2^k, |\xi|).\]
Note that $\widetilde A(\cdot-\frac{2^ky}{t})$ is supported in $B(0,
7/4)$. By integration by part we get the desired inequality.
\end{proof}

\

Now we break
\[\widetilde { E_\psi} P_kf( t \cdot , t) =(I-\widetilde P_k) \widetilde { E_\psi} P_kf( t \cdot , t)
+ \widetilde P_k\widetilde { E_\psi} P_kf( t \cdot , t),\] where
$\widetilde P_k$ is a projection operator defined by $\mathcal F
(\widetilde P_k f)=\widetilde \beta(2^{-k}|\xi|) \widehat f(\xi)$
with $\widetilde \beta\in C_0^\infty(1/2B, 2B))$ satisfying
$\widetilde \beta=1$ on $(B^{-1}, B)$. By integration by parts with
\eqref{est}, it follows that if $1/2\le |y|\le 2$
\[|\mathcal F^{-1}\big((1-\widetilde \beta(2^{-k}|\cdot|))\,m(k,y,\cdot)\big)|\le C
2^{-Mk}(1+|x|)^{-M}\] for any $M$.  Hence, by \eqref{ftrans}
  we get
\begin{align*}
|(I-\widetilde P_k) \widetilde { E_\psi} P_kf( t x , t)| &\le C
t^{-\frac d2} 2^{-Mk} (1+|x|)^{-M} \int  |\beta(|y|) f^\vee(2^k y)
|dy.
\end{align*}
Then, by Schwarz's inequality and Plancherel's theorem,
$|(I-\widetilde P_k) \widetilde { E_\psi} P_kf( t x , t)|\le C
t^{-\frac d2} 2^{-Mk} (1+|x|)^{-M} \|P_kf\|_2$ and thus
\[\|\chi_k^c(t) (I-\widetilde P_k) \widetilde { E_\psi} P_kf(t\cdot,t)\|_{L^r_x}\le C t^{-\frac d2}2^{-Mk} \|P_kf\|_2.\]

Using this and  Littlewood-Paley inequality,  we see that
\begin{equation}\label{string}\begin{aligned} &\Big\| \sum_{k=1}^\infty \chi_{k}^c  \widetilde
{ E_\psi} P_k f (\cdot, t) \Big\|_{L^r_{x}}=t^{\frac{d}r}
\Big\|\sum_{k=1}^\infty \chi_{k}^c \widetilde { E_\psi} P_k f
(t\,\cdot, t) \Big\|_{L^r_x}
\\
\le & \,t^{\frac{d}r}\Big\| \sum_{k=1}^\infty \chi_{k}^c \widetilde
P_k \widetilde { E_\psi} P_k f (t\,\cdot, t) \Big\|_{L^r_x}+
t^{\frac{d}r} \Big\| \sum_{k=1}^\infty \chi_{k}^c (I-\widetilde P_k)
\widetilde { E_\psi} P_k f (t\,\cdot, t) \Big\|_{L^r_x}
\\
\le &\, C t^{\frac{d}r} \chi_{k}^c(t)\Big\| \Big(\sum_{k=1}^\infty |
\widetilde { E_\psi} P_k f (t\,\cdot, t)|^2\Big)^\frac12
\Big\|_{L^r_x} +C t^{\frac{d}r-\frac d2} \sum_{k=1}^\infty 2^{-Mk} \|f\|_2.
 \end{aligned}\end{equation}
Since $q,r\ge 2$, taking integration in $t$, by Minkowski's
inequality we get
\begin{align*} \Big\| \sum_{k=1}^\infty \chi_{k}^c   \widetilde { E_\psi} P_k f (\cdot, t)
&\Big\|_{L^q_{t}L_x^r}\le C \Big(\sum_{k=1}^\infty \Big\| t^\frac
dr\chi_{k}^c\| \widetilde { E_\psi} P_k f (t\,\cdot,
t)\|_{L^r_{x}}\Big\|_{L^q_{t}}^2\Big)^\frac12 \\&+ \sum_{k=1}^\infty
2^{-Mk} \|P_kf\|_2 \|t^{\frac dr-\frac d2}\|_{q} .\end{align*} The
second term in the right hand side is clearly bounded by
$C\|f\|_{H^s}$ if $d/r+1/q<d/2$. Therefore we are reduced to showing
that
\begin{equation}\label{final}
\Big(\sum_{k=1}^\infty \Big\| t^\frac dr\chi_{k}^c\| \widetilde {
E_\psi} P_k f (t\,\cdot, t)\|_{L^r_{x}}\Big
\|_{L^q_{t}}^2\Big)^\frac12\le C\|f\|_{H^s}.
\end{equation}

For this we use the following.

\begin{lem} If $t\ge B2^{2k}$ for some large $B>0$, then for $1\le r\le \infty$
\begin{equation}\label{lplp}\| \widetilde{E_\psi} P_k f (t\,\cdot, t)\|_{L^r_{x}}\le Ct^{-\frac d2} \|P_k f\|_r.\end{equation}
\end{lem}

\begin{proof}
Since $(P_k f)^\vee$ is supported in $\{y:2^{k-1}\le |y|\le
2^{k+1}\} $, we may insert a harmless smooth function $\beta_\circ$
so that
\[\widetilde { E_\psi} P_kf(x,t)= e^{-it\psi(\frac xt)} t^{-\frac d2}\int
e^{it\psi(\frac{x-y}{t})} \widetilde A(\frac{x-y}{t})
\beta_\circ(2^{-k} |y| ) (P_k f)^\vee(y) dy,
\]
where $\beta=\beta_\circ\beta$  and $\beta_\circ$ is supported in
$[2^{-2}, 2^2]$. By rescaling we have
\[ \widetilde { E_\psi} P_kf(tx,t)= t^{-\frac d2}\int \mathcal K(x,z,k) P_kf(z) dz,
\]
where
\[\mathcal K(x,z,k)= (2\pi)^{-d} 2^{dk}e^{-it\psi(x)}\int
e^{it\psi(x-\frac{2^ky}{t})} e^{i2^k z \cdot y} \widetilde
A(x-\frac{2^ky}{t})  \beta_\circ (|y|) dy.\] Since $|y|\sim 1$ and $
t\ge B2^{2k}$,  considering the phase part of this integral, we see
that $\nabla_y\Big(t\psi(x-\frac{2^ky}{t})+2^k z y\Big)
=2^k(z-\nabla \psi(x))+O(1).$ Therefore, by integration by parts we
get
\[|\mathcal K(x,z,k)|\le C2^{kd}
(1+2^k|z-\nabla \psi(x)|)^{-N}.\] Since $\widetilde A$ is supported
in $B(0,3/2)$, it follows that  $\supp \mathcal K(\cdot,z,k)\subset
B(0,7/4)$ if $B$ is sufficiently large. From \eqref{psi} $x\to
\nabla\psi(x)$ is a diffeomorphism on $B(0,15/8)$. Hence it is easy
to see $\int |\mathcal K(x,z,k)| dx<C$. Clearly, $\int |\mathcal
K(x,z,k)| dz$ $\le C$.  Then, \eqref{lplp} follows from Young's
inequality.
\end{proof}

We now return to the proof of \eqref{final}.  We  break  $\chi_k ^c$
to $\chi_{[C2^k, C2^{2k}]} + \chi_{[C2^{2k}, \infty)}$ so that
\begin{align*}
 \Big(\sum_{k=1}^\infty \Big\| t^\frac
dr\chi_{k}^c\| \widetilde { E_\psi} &P_k f (t\,\cdot,
t)\|_{L^r_{x}}\Big\|_{L^q_{t}}^2\Big)^\frac12 \le  \Big (
\sum_{k=1}^\infty \Big\| t^\frac dr\chi_{[C2^k, C2^{2k}]} (t)\|
\widetilde { E_\psi}  P_k f (t\,\cdot,
t)\|_{L^r_{x}}\Big\|_{L^q_{t}}^2\Big)^\frac12 \\ &+
\Big(\sum_{k=1}^\infty \Big\| t^\frac dr\chi_{[C2^{2k}, \infty)}\|
\widetilde { E_\psi} P_k f (t\,\cdot,
t)\|_{L^r_{x}}\Big\|_{L^q_{t}}^2\Big)^\frac12.
\end{align*}
By rescaling we note that the first term of right hand side equals
\[
 \Big(\sum_{k=1}^\infty
\| \widetilde { E_\psi} P_k f \|_{_{{{L_t^q([C2^k, C2^{2k}],
L_x^r(\mathbb R^d))}}}}^2\Big)^\frac12.\] By \eqref{interpol} it
follows that this is bounded by $C \Big (\sum_{k=1}^\infty
2^{2k(\frac{d}r+\frac 2q-\frac d2)}\|P_k f \|_{2}^2\Big
)^\frac12\lesssim \|f\|_{H^{s_c}}. $ So, we only need to consider
the second term.  By making use of \eqref{lplp} we see
\begin{align*}
&\Big(\sum_{k=1}^\infty \| t^\frac dr\chi_{[C2^{2k}, \infty)}\|
\widetilde { E_\psi}P_k f (t\,\cdot,
t)\|_{L^r_{x}}\|_{L^q_{t}}^2\Big )^\frac12\le  C
\Big(\sum_{k=1}^\infty \| t^\frac dr t^{-\frac d2} \chi_{[C2^{2k},
\infty)}\|_{L^q_{t}}^2 \|P_k f\|_r^2\Big)^\frac12
\\
&\le  C \Big(\sum_{k=1}^\infty  2^{4k(\frac {d} r+\frac 1q-\frac
d2)} \|P_k f\|_r^2\Big)^\frac12 \le  C \Big(\sum_{k=1}^\infty
2^{2k(\frac {d} r+\frac 2q-\frac d2)} \|P_k f\|_2^2\Big)^\frac12 \le
C\|f\|_{H^{s_c}}.
\end{align*}
For the last inequality we use Bernstein's inequality $\|P_k
f\|_q\le C2^{k(\frac d2-\frac dq)}\|P_kf\|_2$. This completes the
proof of \eqref{final}.

\subsubsection*{Proof of weak type endpoint estimate} Let us
fix $q,r$ such that $2\le q,r<\infty$ and $d/r+1/q=d/2$. The proof
here is a minor modification of that of \eqref{mixed}. So we shall
be brief.

As before  it suffices to show $\norm{ E (
f)}_{{L_t^{q,\infty}((0,\infty),\, L_x^r(\mathbb R^d))}}\les
\norm{f}_{H^{s_c}}.$ Because of \eqref{interpol} it is enough to
prove that
\[ \norm{ E ( f)}_{{L_t^{q,\infty}((1,\infty),\, L_x^r(\mathbb R^d))}}\les
\norm{f}_{H^{s_c}}.\] By \eqref{pw} it follows that
\[ \|E P_{\le0} f\|_{{L_t^{q,\infty}((1,\infty),\,
L_x^r(\mathbb R^d))}}\le  C\|t^{-\frac d2+\frac d
r}\|_{{L_t^{q,\infty}(1,\infty)}} \| f\|_2\le C\|f\|_2\] because
$t^{-\frac d2+\frac d r}=t^{-\frac1q}\in L^{q,\infty}(0,\infty)$.
Breaking $\sum_{k=1}^\infty E P_k f = \sum_{k=1}^\infty
\chi_{k}^\circ  E P_k f + \sum_{k=1}^\infty \chi_{k}^c  E P_k f,$
for the first sum we get the desired bound by the same argument as
before because $L^q\subset L^{q,\infty}$. Hence, it is enough to
show
\begin{equation}\label{reduced1}
\Big\|\sum_{k=1}^\infty \chi_{k}^c  E P_k f
\Big\|_{{L_t^{q,\infty}((1,\infty),\, L_x^r(\mathbb R^d))}} \le
C\|f\|_{H^{s_c}}.
\end{equation}

Decomposing further $\chi_{k}^c  E P_k f= \chi_{k}^c \widetilde E
P_k f+ \chi_{k}^c \mathcal R P_k f$, the contribution from $
\chi_{k}^c  \mathcal  R P_k f$ is controlled by the bound obtained
previously. So it is sufficient to show that
\[\Big\|\sum_{k=1}^\infty \chi_{k}^c  \widetilde { E_\psi} P_k f \Big\|_{{L_t^{q,\infty}((1,\infty),\,
L_x^r(\mathbb R^d))}}\le C\|f\|_{H^{s_c}}.\] Since $q>2$,
$L^{q/2,\infty}$ is a Banach space. Using Minkowski's and triangle
inequalities, from \eqref{string} we get
\begin{align*}
\Big\| \sum_{k=1}^\infty \chi_{k}^c   \widetilde { E_\psi} P_k f
(\cdot, t) &\Big\|_{{L_t^{q,\infty}((1,\infty),\, L_x^r(\mathbb
R^d))}}\le C \Big(\sum_{k=1}^\infty \Big \| t^\frac dr\chi_{k}^c\|
\widetilde { E_\psi} P_k f (t\,\cdot, t)\|_{L^r_{x}}\Big
\|_{L^{q,\infty}_{t}}^2\Big)^\frac12 \\&+ \sum_{k=1}^\infty 2^{-Nk}
\|P_kf\|_2 \|t^{\frac dr-\frac d2}\|_{L^{q,\infty}} .\end{align*}
Since $t^{\frac dr-\frac d2}\in L^{q,\infty}(0,\infty)$, the second
term on the right hand side is bounded by $C\|f\|_{H^{s_c}}$ and
\[\Big\| t^\frac dr\chi_{[C2^{2k}, \infty)}\|
\widetilde { E_\psi}P_k f (t\,\cdot,
t)\|_{L^r_{x}}\Big\|_{L^{q,\infty}_{t}} \le C2^{k(\frac {d} r+\frac
2q-\frac d2)} \|P_k f\|_2.\]
As before this follows from \eqref{lplp} and  Bernstein's
inequality. This completes the proof.

\section{Strichartz estimates: Proofs of Corollaries}

In this section we prove Corollary \ref{str}, \ref{str2} and
\ref{str3}.

\subsubsection*{Proof of Corollary \ref{str}}
Since $q, r \ge 2$,
 using Littlewood-Paley theory and Minkowski's inequality, we have
\begin{align*}
&\qquad\|e^{it(-\Delta)^{ \alpha/ 2}}\varphi\|_{L_t^q(\mathbb R,\,
L_x^r(\mathbb R^d))}^2  \lesssim \sum_{k \in \mathbb Z}
\|e^{it(-\Delta)^{\alpha/2}} P_k \varphi\|_{L_t^q(\mathbb R,\,
L_x^r(\mathbb R^d))}^2.
\end{align*}
We  observe that $ e^{it(-\Delta)^{ \alpha/ 2}}P_kg(x) =
2^{dk}e^{it2^{\alpha k}(-\Delta)^{ \alpha/ 2}}P_0g_{2^{k}}(2^kx),$
where $g_{\lambda}(x) = \lambda^{-d}g( x/\lambda). $ Since $0<s_c <
1/2$, recalling \textit{Remark}\,\ref{remark1}, from rescaling and
Theorem \ref{xed}  we get
\begin{align*}
\|e^{it(-\Delta)^{ \alpha/ 2}}\varphi\|_{L_t^q(\mathbb R,\,
L_x^r(\mathbb R^d))}^2 &\lesssim \sum_{k \in \mathbb
Z}2^{2k(d-\frac{d}{r}-\frac{\alpha}{q})}\|e^{it(-\Delta)^{\alpha/2}}P_0
\varphi_{2^{k}}\|_{L_t^q(\mathbb R,\,
L_x^r(\mathbb R^d))}^2\\
&\lesssim \sum_{k \in \mathbb Z}
2^{2k(d-\frac{d}{r}-\frac{\alpha}{q})}\|\beta\widehat
\varphi(2^k\cdot)\|_{\dot H^{s_c}}^2.
\end{align*}

By Plancherel's theorem and rescaling (note that
$P_0\varphi_{2^k}=2^{-kd}(P_k\varphi)(2^{-k}x)$) it follows that
\begin{align*}
\|e^{it(-\Delta)^{ \alpha/ 2}}\varphi\|_{L_t^q(\mathbb R,\,
L_x^r(\mathbb R^d))}^2 &\lesssim \sum_{k \in \mathbb Z}
2^{2k(d-\frac{d}{r}-\frac{\alpha}{q})} \||x|^{s_c} P_0
\varphi_{2^{k}}\|_2^2
\\&\lesssim \sum_{k \in \mathbb Z}
2^{2k(2-\alpha)/q}\||x|^{s_c} P_k \varphi\|_{2}^2.
\end{align*}

We define $\mathcal P_k$ by $\mathcal F( \mathcal
 P_k f)= \beta_\circ(2^{-k}|\cdot|)\widehat f$ so that $\mathcal
P_k P_k=P_k$. (Here $\beta_\circ$ is a smooth function supported in
 $[2^{-2}, 2^2]$ such that $\beta_\circ\beta=\beta$.)
 Also we set $\widetilde P_k= 2^{k\frac{(2-\alpha)}q} (-\Delta)^{\frac{\alpha-2}{2q}} \mathcal P_k.$
 Hence we get
\[\|e^{it(-\Delta)^{ \alpha/
2}}\varphi\|_{L_t^q(\mathbb R,\, L_x^r(\mathbb R^d))}^2 \lesssim
\int \big(\sum_{k \in \mathbb Z} |\widetilde P_k P_k
(-\Delta)^{\frac{2-\alpha}{2q}} \varphi|^2\big) |x|^{2s_c} dx.\]
Since $0<s_c < 1/2$, $|x|^{2s_c}$ is an $A_2$-weight (see
\cite[p.219]{st2}). Thus by a vector valued inequality for $A_p$
weight (e.g. \cite[Remarks 6.5, p. 521]{garu}) it follows that
\[\|e^{it(-\Delta)^{ \alpha/
2}}\varphi\|_{L_t^q(\mathbb R,\, L_x^r(\mathbb R^d))}^2 \lesssim\int
\big(\sum_{k \in \mathbb Z} |P_k (-\Delta)^{\frac{2-\alpha}{2q}}
\varphi|^2\big)|x|^{2s_c} dx.\] By Littlewood-Paley theory (e.g.
\cite[p. 275]{Ru}, \cite{tri}) in weighted $L^p$ spaces the right
hand side  is bounded by
\[C\int |(-\Delta)^{\frac{2-\alpha}{2q}}\varphi|^2|x|^{2s_c} dx.\]
Therefore we get the desired
inequality.
\qed

\subsubsection*{Proof of Corollary \ref{str2}} In order to prove Corollary \ref{str2}, it is
sufficient to show that
\[  \|e^{it\sqrt{-\Delta}}\varphi\|_{L_t^q(\mathbb R,\, L_x^r(\mathbb R^d))} \le C\||x|^{s_c^w} \varphi\|_{2}\]
whenever $\widehat \varphi$ is supported in $\{\xi: 1/2< |\xi|<2\}$.
Once it is established, the rest of proof is identical with that of
Corollary \ref{str}. By a finite decomposition, rotation and
rescaling we may assume that $\widehat f$ is supported in
$\Gamma=\{\xi=(\bar \xi,\xi_d): |\bar \xi|< \xi_d/100,\,1/2< \xi_d<2
\}$. Let us set
\[\mathcal T\varphi(x,t)
=\int_\Gamma e^{ix\cdot \xi+ it(|\xi|-\xi_d) } \widehat \varphi(\xi)
d\xi=\int_\Gamma e^{i\bar x\cdot \bar \xi+ i x_d \xi_d+
it\xi_d\theta(\bar \xi/\xi_d) } \widehat \varphi(\xi) d\xi\] with
$\theta(\eta)=\sqrt{1+|\eta|^2}-1$. Then, by a simple change of
variables $x_d\to x_d-t$ it is enough to show that
\begin{equation}\label{wave} \|\mathcal T\varphi\|_{L_t^q(\mathbb R,\,
L_x^r(\mathbb R^d))} \le C\||\bar x|^{s_c^w}
\varphi\|_{2}\end{equation} provided that $\supp\widehat
\varphi\subset \Gamma$. By the Hausdorff-Young inequality in $x_d$
and Minkowski's inequality, the left hand side is bounded by
\begin{align*}
C\Big(\int_{1/2}^2\Big\| \int_{|\bar\xi|\le \frac1{50}}  e^{i\bar
x\cdot \bar \xi+ it\xi_d\theta(\bar \xi/\xi_d)} \widehat
\varphi(\bar\xi, \xi_d)d\bar\xi\Big\|_{L_t^q(\mathbb R,\, L_{\bar
x}^r(\mathbb R^{d-1}))}^{r'} d\xi_d\Big)^\frac1{r'}.
\end{align*}
Freezing $\xi_d\in (1/2,2)$, the Hessian matrix of
$\theta(\cdot/\xi_d)$ is non singular. So, we apply  Theorem
\ref{xed} to the extension operator defined by
$\theta(\cdot/\xi_d)$.  In fact, since $\theta(\eta)$ is close to
$\frac12|\eta|^2$ and $1/2\le \xi_d\le 2$, it is easy to see that
there is a uniform bound $C$ independent of $\xi_d$ so that
\[\Big\| \int_{|\bar\xi|\le \frac1{50}}  e^{i\bar x\cdot \bar \xi+
it\xi_d\theta(\bar \xi/\xi_d)}  g(\bar \xi)
d\bar\xi\Big\|_{L_t^q(\mathbb R,\, L_{\bar x}^r(\mathbb
R^{d-1}))}\le C\|(1-\Delta_{\bar x})^{s_c^w/2}  g\|_2.\] Therefore,
recalling \textit{Remark}\,\ref{remark1} and taking integration in $\xi_d$,
we get
\[  \|\mathcal T\varphi\|_{L_t^q(\mathbb R,\, L_x^r(\mathbb R^d))} \le C
\Big(\int_{1/2}^2\Big\| (-\Delta_{\bar x})^{s_c^w/2} \widehat
\varphi(\bar\xi, \xi_d) \Big\|_{L^2}^{r'} d\xi_d\Big)^{\frac
1{r'}}.\] Then \eqref{wave} follows by Plancherel's theorem and
H\"older's inequality. \qed

\subsubsection*{Proof of Corollary \ref{str3}} %
For the proof it suffices to show the case $q=r$. The other cases follow
from interpolation with the estimate
$\|(-\Delta)^{{\gamma_1}(\infty, 2, d)/2}
e^{it\sqrt{-\Delta}}\varphi\|_{L_t^\infty(\mathbb R,\, L_x^2(\mathbb
R^d))}$ $\lesssim C\|\varphi\|_{H^{0}_{sph}}$. By Littlewood-Paley
theory it is enough to show that
\[\|e^{it\sqrt{-\Delta}}\varphi\|_{L_t^q(\mathbb R,\, L_x^r(\mathbb
R^d))}\lesssim C\|\varphi\|_{H^{s_c^w}_{sph}}\] for $\varphi$ of
which Fourier transform is supported in $\{\xi:1/2\le |\xi|\le 2\}$.
For this we write
\[e^{it\sqrt{-\Delta}}\varphi(x)=\frac1{(2\pi)^{d}}\int_{\frac12}^2 \int_{\ball^{d-1}} e^{i\rho  x\cdot\omega +it\rho}
\widehat \varphi(\rho\omega) d\omega \rho^{d-1} d\rho.\] By the
Hausdorff-Young inequality in $t$ and Minkowski's inequality we get
\[\|e^{it\sqrt{-\Delta}}\varphi\|_{L_x^q(\mathbb R^d,\, L_t^q(\mathbb
R))}\lesssim  \| R^*(\widehat \varphi(\rho\cdot))(\rho\cdot)
\|_{L_\rho^{q'}((1/2,2):L_x^q(\mathbb R^d))}.\] Theorem
\ref{thm-sphere} and H\"older's inequality imply
\begin{align*}\|e^{it\sqrt{-\Delta}}\varphi\|_{L_x^q(\mathbb R^d,\,
L_t^q(\mathbb R))}
 &\lesssim  \|(1-\Delta_{\omega}) ^{s_c^w/2}
\widehat \varphi(\rho\omega))
\|_{L_\rho^{2}((1/2,2):L_\omega^2(\mathbb S^{d-1}))} .\end{align*}
Now recalling $(1-\Delta_{\omega}) ^{\nu/2} \widehat g= \mathcal
F((1-\Delta_{\omega}) ^{\nu/2}  g)$, we get the desired inequality
by Plancherel's theorem. \qed

\subsubsection*{Acknowledgment} Y. Cho is supported by NRF grant 2011-0005122 (Republic of
Korea). Z. Guo is supported in part by NNSF of China (No.11371037),
Beijing Higher Education Young Elite Teacher Project (No. YETP0002),
and Fok Ying Tong education foundation (No. 141003). S. Lee is
supported in part by NRF Grant 2012008373 (Republic of Korea). The
second named author would like to thank J. Bourgain for a discussion
about the restriction estimate and his encouragement.

 {\bibliographystyle{plain}

\end{document}